\documentclass[11pt,reqno]{amsart}

\usepackage[T1]{fontenc}
\usepackage[utf8]{inputenc}
\usepackage{amsmath,amssymb,amsthm}
\usepackage{mathpazo}
\usepackage[margin=1in]{geometry}
\usepackage{microtype}
\usepackage{hyperref}
\hypersetup{hypertex=true,
	colorlinks=true,
	linkcolor=blue,
	anchorcolor=blue,
	citecolor=blue }

\numberwithin{equation}{section}
\newtheorem{theorem}{Theorem}[section]
\newtheorem{lemma}[theorem]{Lemma}
\newtheorem{proposition}[theorem]{Proposition}
\theoremstyle{remark}
\newtheorem{remark}[theorem]{Remark}

\newcommand{\R}{\mathbb{R}}
\renewcommand{\d}{\,\mathrm{d}}
\newcommand{\Hr}{H^1_{\mathrm{rad}}}

\title[Radial Robin Lane--Emden solutions]{Radial uniqueness and nondegeneracy\\
for the Robin Lane--Emden equation}
\date{}

\begin{document}

\title[Radial Robin Lane--Emden solutions]{Radial uniqueness and nondegeneracy for the Robin Lane--Emden equation}
	
	\author{Jiaguo Li, Tianqi Xu}

\address{Jiaguo Li: School of Mathematical Sciences, Xiamen University, Xiamen, Fujian, 361005, P. R. China}
	\email{19020251153741@stu.xmu.edu.cn}

\address{Tianqi Xu: School of Mathematical Sciences, Xiamen University, Xiamen, Fujian, 361005, P. R. China}
	\email{19020251153763@stu.xmu.edu.cn}

\begin{abstract}
In this paper, we study positive radial solutions of $-\Delta u=u^p$ in a ball $B_R\subset\R^N$, subject to the Robin boundary condition $\partial_\nu u+\beta u=0$, where $N\ge2$ and $\beta>0$. For every Sobolev subcritical exponent $p>1$, we prove the existence of a unique positive radial solution and its nondegeneracy with respect to radial perturbations. The proof relies on the strict monotonicity of the logarithmic slope of the normalized Lane--Emden profile, which follows from the radial Pohozaev identity. When $N\ge3$, the same argument also recovers the classical critical existence range $0<\beta R<N-2$ and proves radial nondegeneracy throughout this interval. At $\beta R=(N-2)/2$, we exhibit $N$ linearly independent nonradial solutions of the linearized problem.
\end{abstract}

\maketitle

\section{Introduction and main results}\label{s1}

In this paper, we are concerned with the following semilinear elliptic problem:
\begin{equation}\label{1-1}
\left\{\begin{aligned}&-\Delta u=u^p &&\text{in }B_R,\\
&u>0 &&\text{in }B_R,\\
&\partial_\nu u+\beta u=0 &&\text{on }\partial B_R,
\end{aligned}\right.
\end{equation}
where $B_R$ is the ball of radius $R>0$ centered at the origin in $\R^N$, $N\ge2$, and $\nu$ denotes the outward unit normal to $\partial B_R$. Throughout the paper, we assume that $\beta>0$ and $p>1$, and write
\begin{equation}\label{1-2}
p_S(N):=\left\{\begin{aligned}&+\infty &&\text{if }N=2,\\
&\frac{N+2}{N-2} &&\text{if }N\ge3.
\end{aligned}\right.
\end{equation}
By a positive radial solution of \eqref{1-1}, we mean a radial function $u\in C^2(\overline{B_R})$ satisfying \eqref{1-1} pointwise. In the subcritical range, every nontrivial nonnegative radial $H^1(B_R)$ weak solution belongs to this class, as recalled in Section~\ref{s3}.

Shooting methods and Pohozaev identities have long been used to study the uniqueness of positive radial solutions to semilinear elliptic equations. For Dirichlet problems, we refer to Ni and Nussbaum \cite{Ni1985}, Erbe and Tang \cite{Erbe1997}, and Shioji and Watanabe \cite{Shioji2013,Shioji2016}. Kwong and Li \cite{Kwong1992} and Kabeya, Yanagida and Yotsutani \cite{Kabeya2002} developed shooting and canonical-form methods for radial problems with various boundary conditions. Uniqueness results for two-point problems with Robin and nonlinear boundary conditions were obtained by Ma and An \cite{Ma2005a,Ma2005b}. Dai and Fu \cite{Dai2007} established related uniqueness and nondegeneracy results for semilinear second-order ordinary differential equations with Robin boundary conditions.

For the Robin problem in the subcritical range, Dai and Fu \cite{Dai2008} proved uniqueness on smooth bounded domains when $\beta$ is sufficiently small, while Fu and Dai \cite{Fu2008} studied positive solutions on annuli. More recently, Chen, Grossi and Li \cite{Chen2026a} determined the leading-order profile as $\beta\to0^+$ on general smooth bounded domains. On a ball, this profile also follows directly from the radial parametrization; see Remark~\ref{re3-3}.

On a disk, Zeraoulia \cite{Zeraoulia2027} established radial uniqueness and constructed branches of positive nonradial solutions. In a recent preprint, Kenioua \cite{Kenioua2026} further studied the first angular degeneracies of the planar radial branch. Chen, Grossi and Li \cite[Theorem 1.3]{Chen2026b} proved radial uniqueness and radial nondegeneracy throughout the subcritical range in dimension two. In dimensions $N\ge3$, their result assumes $\beta\ge N-2$ on the unit ball, or equivalently $\beta R\ge N-2$ on $B_R$. For each fixed $\beta>0$, they also proved uniqueness and full nondegeneracy on general smooth bounded domains when the exponent is sufficiently close to one \cite[Theorem 1.1]{Chen2026b}.

We prove that radial uniqueness and radial nondegeneracy hold for every $\beta>0$ and every $1<p<p_S(N)$. In dimensions $N\ge3$, this removes the condition $\beta R\ge N-2$ in \cite[Theorem 1.3]{Chen2026b}.

\begin{theorem}\label{th1-1}
Let $N\ge2$, $R>0$, $\beta>0$, and $1<p<p_S(N)$. Then problem \eqref{1-1} has a unique positive radial solution $u_{\beta,R}$. Moreover, this solution is nondegenerate with respect to radial perturbations: if $h\in\Hr(B_R)$ is a weak solution of
\begin{equation}\label{1-3}
\left\{\begin{aligned}&-\Delta h-pu_{\beta,R}^{p-1}h=0 &&\text{in }B_R,\\
&\partial_\nu h+\beta h=0 &&\text{on }\partial B_R,
\end{aligned}\right.
\end{equation}
then $h\equiv0$.
\end{theorem}

To prove the theorem, we express the Robin boundary condition in terms of the normalized Lane--Emden profile. Let $V$ be the regular radial solution with $V(0)=1$ and $V'(0)=0$, and let $\rho_*$ denote its first zero. By scaling, every positive radial solution is obtained from $V$, and the Robin condition reduces to
\[
Z(\tau)=\beta R,\qquad Z(\rho):=-\frac{\rho V'(\rho)}{V(\rho)}.
\]
Combining the equation for $Z$ with the radial Pohozaev identity for the quantity $P$ defined in \eqref{2-4}, we have
\begin{equation}\label{1-4}
\rho Z'(\rho)=\frac{2P(\rho)}{\rho^{N-2}V(\rho)^2}+\frac{p-1}{p+1}\rho^2V(\rho)^{p-1}>0.
\end{equation}
The first zero is finite in the subcritical range, and the endpoint limits are $Z(0^+)=0$ and $Z(\rho_*^-)=+\infty$. Thus $Z$ is a strictly increasing bijection from $(0,\rho_*)$ onto $(0,+\infty)$, and each $\beta>0$ determines a unique scale. Differentiating the dilation family at this scale, we obtain a nonzero regular radial solution of the linearized equation. This derivative spans the space of regular radial solutions and has a nonzero Robin boundary residual by \eqref{1-4}. Hence the radial kernel is trivial.

\begin{remark}
The existence and uniqueness assertions in the above theorem may also follow from the results in \cite{Kabeya2002} through suitable but non-obvious transformations and additional arguments. We became aware of this possible connection only after completing the present work, during a further literature search and comparison of results undertaken with the assistance of OpenAI tools.
\end{remark}

The same approach applies at the critical exponent, where the radial existence criterion was proved by Wang \cite[Theorem 4.2]{Wang1991}. Radial uniqueness in dimension three is contained in \cite{Kabeya2001}, while uniqueness and radial nondegeneracy in dimensions $N\ge4$ are treated in \cite[Section 5]{Kabeya2006}. The explicit boundary calculation used in \cite[Lemma 5.2]{Kabeya2006} also applies when $N=3$.

For completeness and comparison with Theorem~\ref{th1-1}, we collect these facts in the following theorem; no novelty is claimed for this critical result. We state it in our notation and include a short proof based on the explicit normalized profile, whose logarithmic slope increases from $0$ to $N-2$.

\begin{theorem}\label{th1-2}
Let $N\ge3$, $R>0$, $\beta>0$, and $p=p_S(N)$. Problem \eqref{1-1} has a positive radial solution if and only if
\begin{equation}\label{1-5}
0<\beta R<N-2.
\end{equation}
Whenever it exists, this solution is unique among positive radial solutions and is nondegenerate with respect to radial perturbations.
\end{theorem}

At $\beta R=(N-2)/2$, varying the center and scale of a critical bubble together preserves the Robin boundary condition. The derivatives of this family at the radial solution are $N$ linearly independent elements of the full linearized kernel. We verify this in Remark~\ref{re4-1}. For related classification questions on the unit ball, we refer to Dou, Hu and Xu \cite{Dou2027}.

The paper is organized as follows. In Section~\ref{s2}, we establish the properties of the normalized profile needed for the shooting argument. Section~\ref{s3} contains the proof of Theorem~\ref{th1-1}, the limiting behavior of the radial branch, and an alternative proof of radial nondegeneracy. The critical case is treated in Section~\ref{s4}.

\section{The normalized Lane--Emden profile}\label{s2}

Throughout this section, we fix $N\ge2$ and $1<p<p_S(N)$. We shall prove that the normalized profile has a finite, simple first zero and that its logarithmic slope is strictly increasing before this zero. We use $\rho$ for the radial variable in the normalized equation and reserve $r=|x|$ for the radius in $B_R$.

Let $V$ be the unique regular solution of
\begin{equation}\label{2-1}
V''+\frac{N-1}{\rho}V'+V^p=0,\qquad V(0)=1,\qquad V'(0)=0,
\end{equation}
on its maximal interval of positivity. At the origin, the equation is understood through the integral identity
\begin{equation}\label{2-2}
-V'(\rho)=\rho^{1-N}\int_0^\rho s^{N-1}V(s)^p\d s.
\end{equation}
Local existence and uniqueness follow from the corresponding Volterra equation
\[
V(\rho)=1-\int_0^\rho t^{1-N}\int_0^t s^{N-1}V(s)^p\d s\d t.
\]
In particular, as $\rho\to0^+$, we have
\begin{equation}\label{2-3}
V(\rho)=1-\frac{\rho^2}{2N}+O(\rho^4),\qquad V'(\rho)=-\frac{\rho}{N}+O(\rho^3).
\end{equation}
Set
\[
\rho_*:=\sup\{\rho>0:V(s)>0\text{ for }0\le s<\rho\}\in(0,+\infty].
\]
By \eqref{2-2}, $V' <0$ on $(0,\rho_*)$.

The radial Pohozaev identity is useful here both for proving that $\rho_*$ is finite and for determining the sign of $Z'$. Following \cite{Pohozaev1965}, define
\begin{equation}\label{2-4}
P(\rho):=\rho^N\left(\frac{V'(\rho)^2}{2}+\frac{V(\rho)^{p+1}}{p+1}\right)+\frac{N-2}{2}\rho^{N-1}V(\rho)V'(\rho).
\end{equation}

\begin{lemma}\label{le2-1}
For every $0<\rho<\rho_*$, we have
\begin{equation}\label{2-5}
P'(\rho)=\left(\frac{N}{p+1}-\frac{N-2}{2}\right)\rho^{N-1}V(\rho)^{p+1}>0.
\end{equation}
Consequently, $P(\rho)>0$ on $(0,\rho_*)$.
\end{lemma}

\begin{proof}
Differentiating \eqref{2-4} and using \eqref{2-1}, we obtain \eqref{2-5}. By \eqref{2-3}, $P(0):=\lim_{\rho\to0^+}P(\rho)=0$. Integrating \eqref{2-5}, we have
\begin{equation}\label{2-6}
P(\rho)=\left(\frac{N}{p+1}-\frac{N-2}{2}\right)\int_0^\rho s^{N-1}V(s)^{p+1}\d s.
\end{equation}
Since $1<p<p_S(N)$, the coefficient is positive. The conclusion follows from $V>0$ on $(0,\rho_*)$.
\end{proof}

We next prove that the first zero is finite. The argument combines the positivity of $P$ with a decay estimate derived from \eqref{2-2}.

\begin{lemma}\label{le2-2}
The first zero of $V$ is finite and simple. More precisely,
\begin{equation}\label{2-7}
\rho_*<+\infty,\qquad V(\rho_*)=0,\qquad V'(\rho_*)<0.
\end{equation}
\end{lemma}

\begin{proof}
Suppose, to the contrary, that $V$ remains positive on $(0,+\infty)$. By \eqref{2-2} and the monotonicity of $V$, we have
\begin{equation}\label{2-8}
-V'(\rho)\ge\rho^{1-N}V(\rho)^p\int_0^\rho s^{N-1}\d s=\frac{\rho}{N}V(\rho)^p.
\end{equation}
Integrating this inequality, we obtain
\begin{equation}\label{2-9}
V(\rho)\le\left(1+\frac{p-1}{2N}\rho^2\right)^{-\frac{1}{p-1}}.
\end{equation}
Since $p<p_S(N)$, we have $2(p+1)/(p-1)>N$. In view of \eqref{2-9}, it follows that
\begin{equation}\label{2-10}
\int_0^{+\infty}\rho^{N-1}V(\rho)^{p+1}\d\rho<+\infty,\qquad \rho^{N-2}V(\rho)^2\to0\quad\text{as }\rho\to+\infty.
\end{equation}

Multiplying \eqref{2-1} by $\rho^{N-1}V$ and integrating over $(0,L)$, we have
\begin{equation}\label{2-11}
\int_0^L\rho^{N-1}V'(\rho)^2\d\rho=\int_0^L\rho^{N-1}V(\rho)^{p+1}\d\rho+L^{N-1}V(L)V'(L).
\end{equation}
The boundary term at the origin vanishes by \eqref{2-3}, and the term at $L$ is negative. Combining \eqref{2-10} and \eqref{2-11}, we therefore obtain
\[
\int_0^{+\infty}\rho^{N-1}V'(\rho)^2\d\rho<+\infty.
\]
Hence the nonnegative function
\[
G(\rho):=\rho^{N-1}\bigl(V'(\rho)^2+V(\rho)^{p+1}\bigr)
\]
belongs to $L^1(1,+\infty)$. It follows that $\liminf_{\rho\to+\infty}\rho G(\rho)=0$, since otherwise $G(\rho)\ge c/\rho$ for all sufficiently large $\rho$. Thus there exists a sequence $\rho_j\to+\infty$ such that
\begin{equation}\label{2-12}
\rho_j^N\bigl(V'(\rho_j)^2+V(\rho_j)^{p+1}\bigr)\to0.
\end{equation}
Moreover, by \eqref{2-10} and \eqref{2-12},
\[
\rho_j^{N-1}|V(\rho_j)V'(\rho_j)|=\bigl(\rho_j^{N-2}V(\rho_j)^2\bigr)^{1/2}\bigl(\rho_j^NV'(\rho_j)^2\bigr)^{1/2}\to0.
\]
It follows that $P(\rho_j)\to0$. This contradicts Lemma~\ref{le2-1}, since $P(\rho_j)\ge P(\rho_0)>0$ for any fixed $\rho_0>0$ and all sufficiently large $j$. Hence $\rho_*<+\infty$.

It remains to examine $V$ at $\rho_*$. Since $0<V\le1$ and $V$ is decreasing, $V(\rho_*^-)$ exists. By \eqref{2-2}, $V'(\rho_*^-)$ also exists and is finite. If $V(\rho_*^-)>0$, ordinary continuation for \eqref{2-1} would extend $V$ as a positive solution beyond $\rho_*$, contrary to its definition. Thus $V(\rho_*)=0$. Passing to the limit in \eqref{2-2}, we obtain
\begin{equation}\label{2-13}
-\rho_*^{N-1}V'(\rho_*)=\int_0^{\rho_*}s^{N-1}V(s)^p\d s>0,
\end{equation}
which proves that the zero is simple.
\end{proof}

We now turn to the logarithmic slope, since the Robin condition prescribes its value at the boundary. For $0<\rho<\rho_*$, set
\begin{equation}\label{2-14}
Z(\rho):=-\frac{\rho V'(\rho)}{V(\rho)},\qquad Y(\rho):=\rho^2V(\rho)^{p-1}.
\end{equation}
Both functions are positive. By \eqref{2-1}, we have
\begin{equation}\label{2-15}
\rho Z'=Z^2-(N-2)Z+Y.
\end{equation}

\begin{proposition}\label{pro2-3}
For every $0<\rho<\rho_*$, we have
\begin{equation}\label{2-16}
\rho Z'(\rho)=\frac{2P(\rho)}{\rho^{N-2}V(\rho)^2}+\frac{p-1}{p+1}Y(\rho)>0.
\end{equation}
Moreover, $Z:(0,\rho_*)\to(0,+\infty)$ is a strictly increasing bijection.
\end{proposition}

\begin{proof}
By \eqref{2-4} and \eqref{2-14}, we have
\begin{equation}\label{2-17}
\frac{P(\rho)}{\rho^{N-2}V(\rho)^2}=\frac12\bigl(Z(\rho)^2-(N-2)Z(\rho)\bigr)+\frac{Y(\rho)}{p+1}.
\end{equation}
Combining \eqref{2-15} and \eqref{2-17}, we obtain \eqref{2-16}. Its right-hand side is positive by Lemma~\ref{le2-1}.

It remains to determine the range of $Z$. By \eqref{2-3},
\begin{equation}\label{2-18}
Z(\rho)=\frac{\rho^2}{N}+O(\rho^4)\qquad\text{as }\rho\to0^+.
\end{equation}
Thus $Z(0^+)=0$. At the other endpoint, Lemma~\ref{le2-2} implies that $V(\rho)\to0$ and $V'(\rho)\to V'(\rho_*)<0$, whence $Z(\rho)\to+\infty$. Together with the strict monotonicity, these limits prove the assertion.
\end{proof}

\section{The subcritical Robin problem}\label{s3}

Every nontrivial nonnegative radial weak solution of \eqref{1-1} is classical. Indeed, let $u\in\Hr(B_R)$ be such a weak solution, so that
\[
\int_{B_R}\nabla u\cdot\nabla\varphi\d x+\beta\int_{\partial B_R}u\varphi\d S=\int_{B_R}u^p\varphi\d x\qquad\text{for all }\varphi\in H^1(B_R).
\]
If $N\ge3$, the Sobolev embedding and $p<p_S(N)$ imply that $u^{p-1}\in L^q(B_R)$ for some $q>N/2$. The same conclusion holds in dimension two by choosing a sufficiently large finite Sobolev exponent. Standard iteration applied to $-\Delta u=(u^{p-1})u$ with the Robin condition yields $u\in L^\infty(B_R)$. By elliptic regularity, $u\in C^{2,\alpha}(\overline{B_R})$ for some $\alpha\in(0,1)$. The strong maximum principle implies that $u>0$ in $B_R$. Moreover, $u(R)>0$: otherwise the boundary condition would imply $u'(R)=0$, and uniqueness for the radial Cauchy problem at $r=R$ would force $u\equiv0$.

We can now use scaling to describe all positive radial solutions in terms of $V$. Set
\begin{equation}\label{3-1}
\sigma:=\frac{2}{p-1},\qquad U_\tau(r):=\left(\frac{\tau}{R}\right)^\sigma V\left(\frac{\tau r}{R}\right),\qquad 0\le r\le R,
\end{equation}
where $0<\tau\le\rho_*$. Since $\sigma+2=p\sigma$, $U_\tau$ satisfies $-\Delta U_\tau=U_\tau^p$ in $B_R$. It is positive on $\overline{B_R}$ if $\tau<\rho_*$ and vanishes on $\partial B_R$ if $\tau=\rho_*$.

\begin{lemma}\label{le3-1}
Every radial function $u\in C^2(\overline{B_R})$ satisfying $-\Delta u=u^p$ and $u>0$ in $B_R$ is equal to $U_\tau$ for a unique $\tau\in(0,\rho_*]$. Moreover, $\tau=\rho_*$ if and only if $u(R)=0$.
\end{lemma}

\begin{proof}
Let $a=u(0)>0$. Scaling and uniqueness for the regular radial initial-value problem imply that
\begin{equation}\label{3-2}
u(r)=aV\bigl(a^{(p-1)/2}r\bigr).
\end{equation}
Taking $\tau=Ra^{(p-1)/2}$, we recover \eqref{3-1}. Positivity in $B_R$ implies that $\tau\le\rho_*$, with equality precisely when $u(R)=0$. Finally, $U_\tau(0)=(\tau/R)^\sigma$, so distinct values of $\tau$ correspond to distinct central values.
\end{proof}

For a differentiable radial function $w$, write
\[
\mathcal{B}_{\beta,R}w:=w'(R)+\beta w(R).
\]
By \eqref{2-14} and \eqref{3-1}, the Robin residual of $U_\tau$ is
\begin{equation}\label{3-3}
F_{\beta,R}(\tau):=\mathcal{B}_{\beta,R}U_\tau=\frac1R\left(\frac{\tau}{R}\right)^\sigma V(\tau)\bigl(\beta R-Z(\tau)\bigr),\qquad 0<\tau<\rho_*.
\end{equation}
Thus the boundary condition is equivalent to $Z(\tau)=\beta R$.

For radial nondegeneracy, we shall differentiate \eqref{3-3} with respect to $\tau$. We first check that a radial weak solution of the interior linearized equation is regular at the origin and is uniquely determined by its value there.

\begin{lemma}\label{le3-2}
Let $p>1$ and let $u\in C^2(\overline{B_R})$ be radial and positive in $B_R$. Every $h\in\Hr(B_R)$ satisfying
\[
-\Delta h-pu^{p-1}h=0\qquad\text{weakly in }B_R
\]
has a $C^2$ representative near the origin, satisfies $h'(0)=0$, and is uniquely determined by $h(0)$. In particular, the space of regular radial solutions of this interior equation has dimension at most one.
\end{lemma}

\begin{proof}
The coefficient $pu^{p-1}$ is bounded on $B_R$ and is $C^2$ near the origin. By interior elliptic estimates followed by Schauder estimates, $h$ has a $C^2$ representative near the origin. Since $h$ is radial, we have $h'(0)=0$. Integrating the radial equation, we obtain
\begin{equation}\label{3-4}
h'(r)=-pr^{1-N}\int_0^r s^{N-1}u(s)^{p-1}h(s)\d s.
\end{equation}
Suppose that $h(0)=0$. Set $M(r):=\max_{0\le s\le r}|h(s)|$ and $Q:=\|u^{p-1}\|_{L^\infty(B_R)}$. By \eqref{3-4}, for $0<t\le r$ we have
\[
|h'(t)|\le\frac{pQ}{N}tM(r),\qquad |h(t)|\le\frac{pQ}{2N}t^2M(r).
\]
Hence $M(r)\le pQr^2M(r)/(2N)$, so $M(r)=0$ for all sufficiently small $r$. Ordinary uniqueness for the radial equation then yields $h\equiv0$ in $B_R$. Applying this argument to the difference of two solutions with the same value at the origin completes the proof.
\end{proof}

\begin{proof}[Proof of Theorem~\ref{th1-1}]
By Proposition~\ref{pro2-3}, there exists a unique
\begin{equation}\label{3-5}
\tau_{\beta,R}:=Z^{-1}(\beta R)\in(0,\rho_*).
\end{equation}
In view of \eqref{3-3}, $u_{\beta,R}:=U_{\tau_{\beta,R}}$ satisfies the Robin boundary condition. Conversely, by Lemma~\ref{le3-1}, every positive radial solution of \eqref{1-1} has the form $U_\tau$. The endpoint $\tau=\rho_*$ is impossible, since $U_{\rho_*}(R)=0$ and $U_{\rho_*}'(R)<0$ by Lemma~\ref{le2-2}. Hence $0<\tau<\rho_*$, and \eqref{3-3} implies that $\tau=\tau_{\beta,R}$. This proves existence and uniqueness.

It remains to prove radial nondegeneracy. We set
\begin{equation}\label{3-6}
\Psi(r):=\left.\partial_\tau U_\tau(r)\right|_{\tau=\tau_{\beta,R}}.
\end{equation}
Differentiating the identity $-\Delta U_\tau=U_\tau^p$ at $\tau=\tau_{\beta,R}$, we have
\begin{equation}\label{3-7}
-\Delta\Psi=pu_{\beta,R}^{p-1}\Psi\quad\text{in }B_R,\qquad \Psi(0)=\sigma R^{-\sigma}\tau_{\beta,R}^{\sigma-1}\ne0.
\end{equation}
By Lemma~\ref{le3-2}, every radial weak solution of the interior equation in \eqref{1-3} is a multiple of $\Psi$.

Differentiating \eqref{3-3} at $\tau=\tau_{\beta,R}$ and using $Z(\tau_{\beta,R})=\beta R$, we obtain
\begin{equation}\label{3-8}
\mathcal{B}_{\beta,R}\Psi=F_{\beta,R}'(\tau_{\beta,R})=-\frac1R\left(\frac{\tau_{\beta,R}}{R}\right)^\sigma V(\tau_{\beta,R})Z'(\tau_{\beta,R})<0.
\end{equation}
Let $h\in\Hr(B_R)$ be a weak solution of \eqref{1-3}. Then $h=c\Psi$ for some constant $c$. Since $\Psi$ is smooth on $\overline{B_R}$, the weak Robin condition becomes $0=c\mathcal{B}_{\beta,R}\Psi$, whence $c=0$ by \eqref{3-8}.
\end{proof}

The relation $Z(\tau_{\beta,R})=\beta R$ also describes the two ends of the radial branch.

\begin{remark}\label{re3-3}
For fixed $N$, $p$, and $R$, the map $\beta\mapsto\tau_{\beta,R}$ is smooth and strictly increasing. By \eqref{2-18} and \eqref{3-5},
\begin{equation}\label{3-9}
\tau_{\beta,R}\sim\sqrt{N\beta R},\qquad u_{\beta,R}(0)\sim\left(\frac{N\beta}{R}\right)^{\frac1{p-1}}\qquad\text{as }\beta\to0^+.
\end{equation}
Since $V(\tau_{\beta,R}r/R)\to1$ uniformly for $0\le r\le R$, the asymptotic formula for $u_{\beta,R}(0)$ in \eqref{3-9} also holds for $u_{\beta,R}(r)$, uniformly in $r\in[0,R]$. This agrees with \cite[Theorem 1.5]{Chen2026a}, since $|\partial B_R|/|B_R|=N/R$. In particular, $u_{\beta,R}\to0$ uniformly on $\overline{B_R}$ as $\beta\to0^+$. At $\beta=0$, integration of the equation would imply $\int_{B_R}u^p\d x=0$, which excludes positive Neumann solutions.

As $\beta\to+\infty$, we have $\tau_{\beta,R}\to\rho_*$ and
\begin{equation}\label{3-10}
u_{\beta,R}(x)\to\left(\frac{\rho_*}{R}\right)^\sigma V\left(\frac{\rho_*|x|}{R}\right)\qquad\text{in }C^2(\overline{B_R}).
\end{equation}
The limit is the unique positive radial Dirichlet solution. To verify the convergence at the origin, it suffices to observe from \eqref{2-2} that $V'(\rho)/\rho$ extends continuously to $\rho=0$. The radial Hessian formula then yields \eqref{3-10} in Cartesian coordinates.
\end{remark}

Combining the boundary calculation in \cite[Lemma 3.1]{Chen2026b} with the positivity of the Pohozaev quantity, we obtain another proof of radial nondegeneracy.

\begin{remark}\label{re3-4}
After rescaling, we may assume that $R=1$. Let $u$ be a positive radial solution of \eqref{1-1}, and suppose that the linearized Robin problem has a nonzero radial solution $h$. By Green's identity applied to $u$ and $h$, we have
\begin{equation}\label{3-11}
(p-1)\int_{B_1}u^ph\d x=0.
\end{equation}
Set $w(r):=ru'(r)$. Differentiating the equation under dilations, we obtain
\[
-\Delta w-pu^{p-1}w=2u^p\qquad\text{in }B_1.
\]
Combining this equation with \eqref{3-11} and applying Green's identity once more, we have
\[
\bigl(w'(1)+\beta w(1)\bigr)h(1)=0.
\]
Here $h(1)\ne0$, since otherwise the Robin condition would imply $h'(1)=0$, and uniqueness for the radial Cauchy problem would force $h\equiv0$. Using the equation for $u$ and the identity $u'(1)=-\beta u(1)$, we conclude that
\begin{equation}\label{3-12}
u(1)^{p-1}=\beta(N-2-\beta).
\end{equation}
This is already impossible when $N=2$ or $\beta\ge N-2$. If $N\ge3$ and $0<\beta<N-2$, define
\[
P_u(r):=r^N\left(\frac{u'(r)^2}{2}+\frac{u(r)^{p+1}}{p+1}\right)+\frac{N-2}{2}r^{N-1}u(r)u'(r).
\]
By the same calculation as in Lemma~\ref{le2-1}, $P_u(r)>0$ for $0<r\le1$. Combining \eqref{3-12} with the Robin condition, we obtain the contradiction
\[
0<\frac{P_u(1)}{u(1)^2}=-\frac{p-1}{2(p+1)}\beta(N-2-\beta)<0.
\]
\end{remark}

\section{The critical case}\label{s4}

At the critical exponent, the normalized profile is explicit. Its logarithmic slope remains strictly increasing, with finite limit $N-2$; this accounts for the parameter restriction in \eqref{1-5}.

\begin{proof}[Proof of Theorem~\ref{th1-2}]
Let $p=(N+2)/(N-2)$. The normalized critical profile is
\begin{equation}\label{4-1}
V_c(\rho)=\left(1+\frac{\rho^2}{N(N-2)}\right)^{-\frac{N-2}{2}}.
\end{equation}
It satisfies \eqref{2-1} and is positive on $[0,+\infty)$. By a direct computation,
\begin{equation}\label{4-2}
Z_c(\rho):=-\frac{\rho V_c'(\rho)}{V_c(\rho)}=\frac{(N-2)\rho^2}{N(N-2)+\rho^2}.
\end{equation}
Thus $Z_c$ is a strictly increasing bijection from $(0,+\infty)$ onto $(0,N-2)$.

By the same scaling argument as in Lemma~\ref{le3-1}, every positive radial critical solution has the form
\begin{equation}\label{4-3}
U_\tau^c(r):=\left(\frac{\tau}{R}\right)^{\sigma_c}V_c\left(\frac{\tau r}{R}\right),\qquad \sigma_c:=\frac{N-2}{2},\qquad \tau>0.
\end{equation}
Its Robin boundary condition is equivalent to $Z_c(\tau)=\beta R$. In view of \eqref{4-2}, this equation has a unique solution if and only if $0<\beta R<N-2$, and in that case
\begin{equation}\label{4-4}
\tau^2=\frac{N(N-2)\beta R}{N-2-\beta R}.
\end{equation}
This proves the existence and uniqueness assertions.

It remains to verify radial nondegeneracy. Let $\tau$ be determined by \eqref{4-4} and set $\Psi_c:=\partial_\tau U_\tau^c$. Differentiating the equation, we have
\[
-\Delta\Psi_c=p(U_\tau^c)^{p-1}\Psi_c\quad\text{in }B_R,\qquad \Psi_c(0)=\sigma_cR^{-\sigma_c}\tau^{\sigma_c-1}\ne0.
\]
By Lemma~\ref{le3-2}, $\Psi_c$ spans the space of regular radial solutions of the interior linearized equation. Set $F_{\beta,R}^c(\tau):=\mathcal{B}_{\beta,R}U_\tau^c$. The same boundary calculation as in \eqref{3-3}, differentiated at $Z_c(\tau)=\beta R$, yields
\begin{equation}\label{4-5}
\mathcal{B}_{\beta,R}\Psi_c=(F_{\beta,R}^c)'(\tau)=-\frac1R\left(\frac{\tau}{R}\right)^{\sigma_c}V_c(\tau)Z_c'(\tau)<0,
\end{equation}
where
\[
Z_c'(\tau)=\frac{2N(N-2)^2\tau}{\bigl(N(N-2)+\tau^2\bigr)^2}>0.
\]
Consequently, every radial $H^1(B_R)$ solution of the linearized Robin problem is a multiple of $\Psi_c$, and its boundary condition forces that multiple to vanish by \eqref{4-5}. This proves Theorem~\ref{th1-2}.
\end{proof}

We finally examine nonradial elements of the linearized kernel at $\beta R=(N-2)/2$.

\begin{remark}\label{re4-1}
For $a\in B_R$ and $\lambda=(R^2-|a|^2)^{1/2}$, consider the critical bubbles \cite{Talenti1976}
\begin{equation}\label{4-6}
U_{\lambda,a}(x):=[N(N-2)]^{\frac{N-2}{4}}\left(\frac{\lambda}{\lambda^2+|x-a|^2}\right)^{\frac{N-2}{2}}.
\end{equation}
These functions solve $-\Delta U=U^{p_S(N)}$ in $\R^N$. On $|x|=R$, using $\lambda^2+|a|^2=R^2$, we have
\begin{equation}\label{4-7}
\frac{\partial_\nu U_{\lambda,a}}{U_{\lambda,a}}=-\frac{(N-2)(R^2-a\cdot x)}{R\bigl(\lambda^2+|x-a|^2\bigr)}=-\frac{N-2}{2R}.
\end{equation}
Hence, for $\beta R=(N-2)/2$, the family $a\mapsto U_{\sqrt{R^2-|a|^2},a}$ consists of positive solutions of \eqref{1-1}. This family is a reparametrization and rescaling of the family in \cite[Theorem 1.1(iii)]{Dou2027}. They are nonradial about the origin when $a\ne0$ and pass through the radial solution $u_c:=U_{R,0}$ at $a=0$.

Differentiating this family at $a=0$, we obtain the following solutions of the full linearized problem:
\begin{equation}\label{4-8}
\left.\frac{\partial}{\partial a_j}U_{\sqrt{R^2-|a|^2},a}(x)\right|_{a=0}=\frac{(N-2)x_j}{R^2+|x|^2}u_c(x),\qquad j=1,\ldots,N.
\end{equation}
These are $N$ linearly independent first spherical-harmonic modes. Thus $u_c$ is radially nondegenerate by Theorem~\ref{th1-2}, while its full linearized kernel contains the nonradial functions in \eqref{4-8}.
\end{remark}

\bigskip
\noindent{\bf Data Availability.} No datasets were generated or analyzed during the current study.

    \medskip

\noindent{\bf Competing Interests.} The authors declare that they have no competing interests.

\medskip

\noindent{\bf AI declarations.}  The authors used OpenAI models to assist with some calculations, as well as with language editing and refinement of the exposition. The authors take full responsibility for the content of this article.

\end{document}